\documentclass[a4paper,oneside,11pt]{article}
\usepackage[english]{babel}
\usepackage[latin1]{inputenc}
\usepackage{indentfirst}
\usepackage{amssymb}
\usepackage{amsmath}
\usepackage{amsthm}

\Large
\theoremstyle{plain}
 \newtheoremstyle{miestilo}{12pt}{\topsep}{\itshape}{}{\bf}{}{ }{}
\swapnumbers
 \theoremstyle{miestilo}
\newtheorem{theorem}[subsection]{Theorem.}
\newtheorem{proposition}[subsection]{Proposition.}
\newtheorem{lemma}[subsection]{Lemma.}
\newtheorem{corollary}[subsection]{Corollary.}
 \newtheoremstyle{misnotas}{12pt}{12pt}{}{}{\bf}{}{ }{\remark}
 \theoremstyle{misnotas}
 \newtheorem{remark}[subsection]{\ {\bf Remark.}}

 \allowdisplaybreaks[2]

\begin{document}
\flushbottom

 {\LARGE
\centerline{\bf On actions on cubic stochastic matrices
  }}

 \medskip

 \centerline{ Irene Paniello  \footnote{Partially supported  by the
Spanish Ministerio de Ciencia y Tecnolog\'{\i}a and FEDER (MTM2013-45588-C3-2-P)} }

 \medskip
 \centerline{{\sl Department of Statistics and Operative Research
}}
\medskip \centerline{{\sl Public University of Navarre}}
\medskip
 \centerline{{\sl 31006 Pamplona, Spain}}
\medskip
\centerline{E-mail: irene.paniello@unavarra.es}
\medskip



\begin{abstract}
We consider the set of ($n\times n\times n$) cubic stochastic matrices of type (1,2) together with different multiplication rules that not only retain their stochastic
properties but also endow this set with an associative semigroup structure.  Then we introduce different actions of the semigroup of nonnegative column stochastic $n\times n$ matrices on the set of cubic stochastic matrices of type (1,2) and study how  these actions translate to the
cubic matrix slices and marginal distributions. Actions introduced here provide an algebraic framework where considering different changes affecting   the transition probabilities ruling certain biological populations.
\end{abstract}

 \section{Introduction.}

\paragraph{}  Many discrete (physical, biological,...) dynamical systems can be modeled using stochastic processes. Then   interactions occurring
within the system between its particles (also the system evolution between its states) can be assumed to be ruled by a set of transition probabilities, oftentimes arranged into a square $n\times n $ matrix.
However sometimes, depending on the complexity of the system interactions (mainly when interactions involving more than two particles may occur), multi-index matrices, also called tensors, need to be introduced. Independently of their dimension, all these matrices gathering transition probabilities  will be nonnegative and endowed with some stochastic property.

\paragraph{} Introducing multi-index stochastic matrices turns out to be necessary, for instance, when algebraically studying the backward inheritance (from progeny to ancestors) in Mendelian genetic populations, as introduced by Tian and Li \cite{tian_coalgebras}. These authors considered the problem of mathematically formalizing the dynamics of the transference of the genetic information through generations in a biological population ruled by Mendel's laws, focusing on the direction from progeny (offspring) to ancestors. This question was later revisited in \cite{paniello-LAA},   taking advantage of Maksimov's work \cite{m_cubic} and rewritten  in terms of ($n\times n\times n$) cubic matrices.  Assuming that the probabilities ruling the genetic inheritance remain constant through generations and   only a finite number of different genetic types  exists, we can identify such a population's behavior with that of a time homogeneous discrete dynamical system.

 \paragraph{}  In the late 40's Geiringer  \cite{geiringer} remarked that the treatment of the problems arising in biological structures should be comparable  to that
 of any other dynamical system,  whenever the underlying stochastic processes behave similarly.
This paper is aimed at providing an algebraic treatment of those stochastic tensors (multidimensional arrays) appearing in the study of backwards genetic inheritance in Mendelian genetic populations \cite{paniello-LAA,tian_coalgebras}. More precisely, we focus on studying cubic stochastic matrices and their connection to square stochastic matrices.

 \paragraph{}
Cubic stochastic matrices of type (1,2) were first considered by Maksimov in \cite{m_cubic} in an attempt to explain processes in kinetic theory  not representable in terms of Markov chains.
  Maksimov mainly focused on time homogeneous systems, but also considered
  Markov interaction processes (M.i.p.)  \cite[page 62]{m_cubic} involving  nonnegative   stochastic square ($n\times n$) matrices together to cubic   ($n\times n\times n$) stochastic matrices of type (2,3) acting simultaneously.

 \paragraph{}
A similar connection to that occurring in M.i.p.    between  square and cubic matrices
was previously considered
by Lyubich \cite[page 57]{lyubich} in terms of  matrices of mutation rates acting on genetic systems.    Lyubich
studied the  convergence of the so-called evolutionary operators, defined as the composition between a selection and a mutation operator.

 \paragraph{}
The  algebraic approach to   products between stochastic  tensors of different dimensions we attempt here
 is also motivated by Braman \cite{braman}, who  examined actions of (non-necessarily stochastic) $n\times n\times n$ cubic matrices on $n\times n$ square matrices.
 In the current work, however,
 it becomes necessary to take
  into account the stochastic nature of the  matrices  involved, a fact that clearly delimits  which multiplications can be defined since we just look for those  retaining   stochastic properties.

 \paragraph{}
After this introductory section,  we  first review different notions of  square and cubic matrices, all having  stochastic significance, with the aim of setting some notation. In
  the third section, we recall different types of cubic stochastic matrices, as originally introduced by Maksimov in \cite{m_cubic}, to focus on the so-called cubic stochastic matrices of type (1,2).
These cubic stochastic matrices    can gather up more complex transition probabilities than the usual square stochastic matrices,
but also require of multiplication rules able to retain their stochastic properties. Therefore we generalize Maksimov's defined  multiplications for stochastic matrices in an attempt to obtain new tools for studying cubic matrices.  The new multiplication rules come from skewing, or more precisely weighting, the original Maksimov's definition.

 \paragraph{}
In the fourth section we define different actions of the semigroup $ NS(n,\mathbb{R})$, of nonnegative column stochastic $n\times n$ matrices
on the set  ${\rm CS_{(1,2)}(n,\mathbb{R})}   $  of cubic stochastic matrices of type (1,2), endowed with one of the previously defined weighted multiplications.  We review how these actions translate to different lower-dimensional matrices related to cubic stochastic matrices of type (1,2), such as their slices and marginal distributions, as well as how they act on the bivariate Markov chains defined by the marginal distributions.

 \paragraph{}
Finally, in the last section, besides  ending up with a few remarks on some extensions that can be derived from  these types of actions,    we  briefly describe the problem posed by Tian and  Li in \cite{m_cubic} from which our interest in actions defined on cubic matrices stems.

\section{Square stochastic matrices and Markov chains.}

This section is focused on the properties of the set of (nonnegative) column stochastic matrices, stressing on their connection to Markov processes. Examples over different fields are given, before restricting our   attention to matrices with real entries, aimed by their role as transition probabilities ruling stochastic processes.

\subsection{Square stochastic matrices.} \label{square_stochastic}  A  (square) {\sl stochastic matrix} is a $n\times n$ matrix
$\textbf{P}=(p_{ij})_{i,j=1}^n$ such that each column sum equals to one
$$\sum_{i=1}^np_{ij}=1,\quad \hbox{{\rm for all}\  $i,j=1,\ldots,n$}. $$

  This definition corresponds to the so-called column stochastic matrices.  Row stochastic matrices can be defined similarly.
   Vectors and matrices will be written in boldface capital  letters ($ \textbf{A}, \textbf{v}$,...)  and their usual multiplication simply by juxtaposition.

 \begin{lemma}
 The following
 results hold for stochastic matrices over any arbitrary field $F$:
 \begin{enumerate}
  \item[(i)] The product of two stochastic matrices is stochastic.
   \item[(ii)]  The set of stochastic matrices forms a multiplicative semigroup.
    \item[(iii)] The set of nonsingular stochastic matrices forms a (multiplicative) group, called the stochastic group and denoted by $S(n,F)$.
  \end{enumerate}
  \end{lemma}
\begin{proof} (i) is straightforward and (ii) and (iii) follow  from considering the set of stochastic matrices as a subset of  $ M(n,F)$,  the set of all $n\times n$ matrices with entries in $F$.
  \end{proof}

\subsection{Example.} The set  $S(2,\mathbb{R}) $  of $2\times 2$ nonsingular  real  stochastic   matrices, i.e.,
      $$S(2,\mathbb{R}) =   \Big\{ \left(                  \begin{array}{cc}
                    1-a & b \\
                    a &  1-b\\
                  \end{array}
                \right)\mid \  a,b\in \mathbb{R},\    a+b\neq1\Big\},    $$
   is a group, whose Lie geometry is described in \cite{letter}.

\subsection{Examples.}  In \cite{poole} examples of stochastic groups are given for different fields $F$:
  \begin{enumerate}
 \item[(i)]  Given an arbitrary field $F$,   $S(n,F) \cong Aff(n-1,F)$ for all $n\geq 2$, where $Aff(n-1,F)$ denotes the affine group, i.e. the set of all mappings $f:F^{n-1}\to F^{n-1}$ such that $f(\textbf{x}^T)=\textbf{A}\textbf{x}^T+\textbf{b}^T$, with $\textbf{A}\in GL(n-1,F)$, the general linear  group,  and  $\textbf{x}, \textbf{b}\in F^{n-1}$. Then $S(n,F)$     is a group with the   composition.
   \item[(ii)] Given a  Galois field $F$ of order $q=p^m$, $p$ prime,  $S(n,F)$  includes examples such as   $S(2,2)\cong \mathbb{Z}_2$  or $S(2,4)=A_4 $ among many others. Here, following \cite{poole}, we denote  $S(n,F)$ by  $S(n,q)$.
   \end{enumerate}

\begin{remark}
The Markov Lie group was defined by Johnson in \cite{johnson}  to be the subset of the general linear group $GL(n,\mathbb{R})$ preserving the linear form $\sum_{i=1}^n x_i$. As noted in \cite[p. 254]{johnson}  the Markov Lie group is isomorphic to the affine group $ Aff(n-1,\mathbb{R})$.
\end{remark}

\paragraph{ }
More general notions of stochastic matrices can be defined and also endowed with an algebraic structure.

\subsection{Example.}
Let $F$ be $\mathbb{R}$ or $\mathbb{C}$.  An $r$-{\sl generalized doubly stochastic matrix} is a   matrix having all its row and column sums equal to $r\in F$.  The set of all $n\times n$ $r$-generalized stochastic matrices is denoted by $\Omega^r(n,F)$ being $\Omega(n,F)=\cup_r \Omega^r(n,F)$ the set of all generalized doubly stochastic matrices.
  Nonnegative matrices  in  $\Omega^1(n,F)$, i.e. with $r=1$,   are called {\sl doubly stochastic matrices}.
 Moreover  (see \cite{mourad}):
\begin{enumerate}
\item[(i)]
$\Omega^r(n,F)$ is a subalgebra of $M(n, \mathbb{C})$ iff $r=0$.
\item[(ii)]   $GL_1(n,F)=GL(n,F)\cap \Omega^1(n,F)$, where $GL(n,F)$ denotes the general linear group over  $F$,   is a Lie group with Lie algebra $\Omega^0(n,F)$, simple of dimension $(n-1)^2$ over $F$.
\end{enumerate}
Doubly stochastic matrices were  already considered by Sagle  in \cite{sagle} to prove  that the set of doubly stochastic matrices over a field $F$ of characteristic zero is a semisimple associative algebra generated by permutation matrices. Moreover it is isomorphic to $F\oplus M(n-1,F)$.

\paragraph{ }
Interesting subsets of stochastic matrices appear when matrix entries are considered in ordered fields.

\subsection{Example.} Let  $NS(2,\mathbb{R}) $  be the set of   $2\times 2$  (real)  stochastic   matrices with nonnegative entries, i.e.,
      $$NS(2,\mathbb{R}) =   \Big\{ \left(                  \begin{array}{cc}
                    1-a & b \\
                    a &  1-b\\
                  \end{array}
                \right)\mid \  a,b\in \mathbb{R},\   0\leq a,b\leq1\Big\}.    $$
  Matrices in $NS(2,\mathbb{R}) $  are also called stochastic Markov matrices. These matrices provide transition probability matrices in models appearing when studying phylogenetic theory   \cite{letter}.
    Sumner, Fern\'{a}ndez-S\'{a}nchez and Jarvis considered  more general  Markov models in \cite{sumner et al}. We recall here that   Markov models are well-defined subsets  of matrices of  $M(n,\mathbb{C})$  with all column sums equal to one.

\paragraph{ } Since this paper aimed to study those stochastic matrices arising when dealing with transition probabilities, from now on we will restrict to stochastic matrices with nonnegative (real) entries.
 Different results on the structure of the set of nonnegative matrices can be found in \cite{brown,flor,plemmons}.

\begin{remark} We will denote by $NS(n,\mathbb{R}) $ the set of all nonnegative real column stochastic $n\times n $ matrices. Notice that we do not require matrices in $NS(n,\mathbb{R}) $ to be nonsingular.\end{remark}

 \paragraph{ } The following results are straightforward.

\begin{lemma}
  Any convex linear combination of  nonnegative  stochastic matrices is stochastic, i.e., given
  $\textbf{A}, \textbf{B}\in NS(n,\mathbb{R}) $,then $\lambda \textbf{A}+(1-\lambda)\textbf{B}\in NS(n,\mathbb{R}) $ for any $\lambda\in \mathbb{R}$ such that $0\leq \lambda\leq 1$.
  \end{lemma}

\begin{proposition}
    $ NS(n,\mathbb{R}) $ is a multiplicative semigroup.
  \end{proposition}

\paragraph{ }
As it is broadly known, matrices in   $ NS(n,\mathbb{R}) $ are related to Markov processes.

\subsection{Markov chains.}  A {\sl Markov chain} is a discrete time stochastic process  $X=\{X_t,t=0,1,\ldots\}$
 taking values
on a finite or countably infinite state space
 $\{1,2,\ldots\}$  satisfying the Markov property
$$P(X_{t+1}=i_{t+1}\mid X_0=i_0, X_1=i_1,\ldots,X_t=i_t)=P(X_{t+1}=i_{t+1}\mid X_t=i_t,),$$
where $i_k\in\{1,2,\ldots\}$ for $k=0,1,\ldots,t+1$.
Homogeneous Markov processes are those whose    transition
 probabilities $P(X_{t+1}=i_{t+1}\mid X_t=i_t) $ do  not depend on
 time $t$, i.e.
$P(X_{t+1}=i_{t+1}\mid X_t=i_t)=P(X_{t+1+k}=i_{t+1}\mid X_{t+k}=i_t) $ for all $k$.
Assuming the set of states to be finite $\{1,2,\ldots,n\}$ and denoting the transition probabilities by
 $p_{ij}=  P(X_{t+1}=i \mid X_t=j)$,    any homogeneous Markov chain gives rise to a   matrix  $\textbf{P}=(p_{ij})_{i,j=1}^n$ in  $ NS(n,\mathbb{R}) $, i.e., a column stochastic $n\times n$ matrix with nonnegative real entries.

\paragraph{ }
Despite of the wide range of dynamical systems whose evolution can be modeled using Markov chains, multivariate Markov models have also proved to be useful to study dynamical systems, mainly whenever it becomes necessary to simultaneously deal with multiple data sequences. This happens, for instance when modeling DNA sequences or when studying backwards genetic inheritance in Mendelian genetic systems \cite{ching,paniello-MD,tian_coalgebras}.

\subsection{Multivariate Markov models.}\label{MM  models}  Consider $s$ categorical data sequences, each one with $n$ possible states $\{1,2,\ldots,n\}$. Following
\cite{ching} we denote by
 $\textbf{X}^{(j)}_t$ the state probability distribution of the $j$-th sequence at time $t$. Then the
   multivariate Markov model is
$$ \textbf{X}_{t+1}^{(j)} =\sum_{k=1}^s \lambda_{jk} \textbf{P}^{(jk)} \textbf{X}_t^{(k)},\qquad j=1,\ldots,s, $$
  where $\lambda_{jk}\geq0$, $1\leq j,k\leq s$ and
$\sum_{k=1}^s \lambda_{jk}=1$  for $j=1,2,\ldots, s$. Matricially $\textbf{X}_{t+1}=\textbf{Q}\textbf{X}_t$, i.e.
$$\mathbf{X}_{t+1}=\left(%
\begin{array}{l}
 \mathbf{X}_{t+1}^{(1)}  \\
 \mathbf{X}_{t+1}^{(2)} \\
  \vdots  \\
   \mathbf{X}_{t+1}^{(s)}  \\
\end{array}%
\right)%
=\left(%
\begin{array}{cccc}
  \lambda_{11} \textbf{P}^{(11)} & \lambda_{12}\textbf{P}^{(12)}& \cdots & \lambda_{1s} \textbf{P}^{(1s)} \\
 \lambda_{21} \textbf{P}^{(21)} & \lambda_{22}\textbf{P}^{(22)} &\cdots &\lambda_{2s} \textbf{P}^{(2s)}\\
  \vdots & \vdots  &  \ddots &\vdots \\
  \lambda_{s1}  \textbf{P}^{(s1)} &\lambda_{s2}\textbf{P}^{(s2)} & \cdots & \lambda_{ss}\textbf{P}^{(ss)} \\
\end{array}
\right)%
\left(%
\begin{array}{l}
 \mathbf{X}_t^{(1)}  \\
 \mathbf{X}_t^{(2)} \\
  \vdots  \\
   \mathbf{X}_t^{(s)}  \\
\end{array}%
\right)$$
Notice $\textbf{P}^{(jk)}$ is the transition probability matrix from the states in the $k$-th sequence to those in the $j$-th sequence, and therefore, an element in  $ NS(n,\mathbb{R}) $. However column sums $\textbf{Q}$ need not be equal to one.

\paragraph{ } In the following sections we will consider the case $s=2$, i.e., bivariate Markov chains as appear in \cite{paniello-MD} in connection to certain $n\times n\times n$ matrices with stochastic properties.

\subsection{Quadratic stochastic operators.}\label{qso}
Prior to concluding this section it has to be pointed out that, regardless of their many applications, there are still many (biological, physical,...) systems that cannot be described by  Markov models. See for instance \cite{boltzam,dohtani, jenks,kawamura, kolda}. Another comprehensive reference is \cite{mukhamedov} where the stress is on the study of the dynamics of quadratic stochastic operators. As an example to illustrate one of these situations, let us consider quadratic stochastic operators
 acting on the $(n-1)$-dimensional simplex in $\mathbb{R}^n $,
$$S^{n-1}=\Big\{ \textbf{x}=(x_1,\ldots,x_n)\in \mathbb{R}^n\mid x_i\geq 0,\quad \sum_{i=1}^n x_i=1\Big\}$$ i.e.   transformations $V:S^{n-1}\to S^{n-1}$ given by $V(\textbf{x})_k=\sum_{i,j=1}^n p_{ijk}x_ix_j$, $k=1,\ldots,n$, where
$$ p_{ijk}\geq 0,\quad p_{ijk}=p_{jik}\quad \hbox{and}\quad \sum_{k=1}^n p_{ijk}=1,\  \forall\  i,j=1,\ldots,n.$$
Quadratic stochastic operators were introduced by Bernstein  \cite{bernstein} in connection to mathematical models in genetics. Notice that matricially arranging the
  $p_{ijk}$'s we obtain not a square $n\times n$ but a nonnegative cubic $n\times n\times n$ matrix $\textbf{P}=(p_{ijk})_{i,j,k=1}^n$ whose entries have the following biological interpretation: In a genetic population with types $\{1,2,\ldots,n\}$, $p_{ijk}$ denotes the probability of a type $k$ offspring to appear from parental types $i$ and $j$. Here due to the assumption $p_{ijk}=p_{jik}$, parents's type  order is not important.

\section{Cubic stochastic matrices. }

Multi-index  arrays, usually called tensors, make it possible to matricially represent  stochastic processes involving higher-order interaction systems. In this section we focus on third-order
 tensors, i.e.,  elements of $\mathbb{R}^{n_1 \times  n_2\times n_3}$ to consider those having
  $n_1=n_2=n_3$. Following Maksimov's work \cite{m_cubic}, we will name these tensors to be cubic matrices  (see  also \cite{ladra_flow,ladra_ACM})   and use boldface capital letters, e.g. $\textbf{P} $, to denote them, as we did for vectors and matrices (equivalently tensors of order one and two respectively).
Other different terms  (rectangular matrices, cubical third-order tensors,..)
 may also found  to refer to such matrices (see for instance \cite{kolda}).

\subsection{Cubic matrix.}
 A {\sl cubic matrix}
$\textbf{P}=(p_{ijk})_{i,j,k=1}^n$ is an object with three indices
$i,j,k$ which can be uniquely written in the form
$$\textbf{P}=(p_{ijk})_{i,j,k=1}^n=\sum_{i,j,k=1}^np_{ijk}(i,j,k),$$
where  $(i,j,k)$  denotes the  unit cubic matrices, i.e. $(i,j,k)$ is a $n\times n\times n$ cubic matrix whose $(i,j,k)$th entry is equal to 1 and all other entries are equal to 0.

\subsection{Cubic matrix decomposition.}\label{CM decomposition}
Cubic matrices, as any other higher-dimensional tensor, give rise to different subarrays by fixing   subsets of indices.
Following the notation used in \cite{kolda}, we define the (horizontal, lateral and frontal) slices of a cubic matrix
$\textbf{P}$ by fixing two indices and using the colon to indicate all elements (see  \cite[Figure 2.2]{kolda}):
\begin{enumerate}\item [(i)] The {\it horizontal slices} of $\mathbf{P}$  are $\mathbf{P}_{i::}=(p_{ijk})_{j,k=1}^n$, for all $i=1,\ldots,n$.
\item [(ii)] The  {\it lateral slices} of $\mathbf{P}$  are $\mathbf{P}_{:j:}=(p_{ijk})_{i,k=1}^n$, for all $j=1,\ldots,n$.
 \item [(iii)] The {\it frontal slices} of $\mathbf{P}$  are $\mathbf{P}_{::k}=(p_{ijk})_{i,j=1}^n$, for all $k=1,\ldots,n$.
\end{enumerate}
Fixing every index but one, we obtain the following fibers (see   \cite[Fi\-gure~2.1]{kolda}):
\begin{enumerate}\item [(i)] The $jk$-{\it columns} of $\mathbf{P}$  are $\mathbf{P}_{:jk}=(p_{1jk},\ldots, p_{njk})^T $, for all $j,k=1,\ldots,n$.
\item [(ii)] The $ik$-{\it rows} of $\mathbf{P}$  are $\mathbf{P}_{i:k}=(p_{i1k},\ldots, p_{ink})^T$, for all $i,k=1,\ldots,n$.
 \item [(iii)] The $ij$-{\it tubes} of $\mathbf{P}$  are $\mathbf{P}_{ij:}=(p_{ij1},\ldots, p_{ijn})^T$, for all $i,j=1,\ldots,n$.
\end{enumerate}

Slices (and also fibers) can be used to unfold cubic matrices reordering the matrix entries into a rectangular matrix. This process is called {\it matricization } or {\it unfolding}  \cite[2.4]{kolda}. To illustrate this procedure let $\textbf{P}$ be a $3\times 3\times 3$ cubic matrix and let

 $ \mathbf{P}_{::1}=
\left(
\begin{array}{ccc}
  p_{111} &   p_{121} &   p_{131} \\
  p_{211} &   p_{221} &   p_{231} \\
  p_{311} &   p_{321} &   p_{331}
\end{array}
\right) \qquad
$
$
\mathbf{P}_{::2}=
\left(
\begin{array}{ccc}
  p_{112} &   p_{122} &   p_{132} \\
  p_{212} &   p_{222} &   p_{232} \\
  p_{312} &   p_{322} &   p_{332}
\end{array}
\right)
$

\noindent and

$$ \mathbf{P}_{::3}=
\left(
\begin{array}{ccc}
  p_{113} &   p_{123} &   p_{133} \\
  p_{213} &   p_{223} &   p_{233} \\
  p_{313} &   p_{323} &   p_{333}
\end{array}
\right)
$$
be the three frontal slices of $\textbf{P}$. Then one of the possible matricizations of $\textbf{P}$ is the $9\times 3$ matrix given by its frontal slices:

$$
 \left(
   \begin{array}{c|c|c}
   \mathbf{P}_{::1} & \mathbf{P}_{::2} & \mathbf{P}_{::3} \\
   \end{array}
   \right)
   =
   \left(
     \begin{array}{ccc|ccc|ccc}
       p_{111} & p_{121} & p_{131} & p_{112} & p_{122} & p_{132} & p_{113} & p_{123} & p_{133} \\
       p_{211} & p_{221} & p_{231} & p_{212} & p_{222} & p_{232} & p_{213} & p_{223} & p_{233} \\
       p_{311} & p_{321} & p_{331} & p_{312} & p_{322} & p_{332} & p_{313} & p_{323} & p_{333} \\
     \end{array}
   \right)
 $$
As noted in \cite{kolda} the choice of slices or fibers when unfolding tensors is not as important as the reordering remaining consistent with respect to any defined computation.

\subsection{Maksimov's cubic stochastic matrices.}\label{def_cubic stochastic}
  A  cubic matrix
 $\textbf{P}=(p_{ijk})_{i,j,k=1}^n$ is said to be:
 \begin{enumerate}
 \item[(i)]  {\sl stochastic of type (1,2)} if
 $p_{ijk}\geq0$ and $\sum_{i,j=1}^np_{ijk}=1$ for all $k=1,\ldots,n$.
  \item[(ii)]  {\sl stochastic of type (2,3)} if
 $p_{ijk}\geq0$ and $\sum_{j,k=1}^np_{ijk}=1$ for all $i=1,\ldots,n$.
 \item[(iii)]  {\sl stochastic of type (1,3)} if
 $p_{ijk}\geq0$ and $\sum_{i,k=1}^np_{ijk}=1$ for all $j=1,\ldots,n$.
  \end{enumerate}
These matrices were first introduced by Maksimov \cite{m_cubic} to model physical systems
 such that interactions within the system could not be described in terms of Markov processes. Maksimov-like definitions can be extended to describe different classes of (always nonnegative) cubic stochastic matrices.

\subsection{3-stochastic cubic matrices.}\label{3 SCM}  A cubic matrix
 $\textbf{P}=(p_{ijk})_{i,j,k=1}^n$ is said to be 3-{\it stochastic} if if
 $p_{ijk}\geq0$ and $\sum_{k=1}^np_{ijk}=1$ for all $i,j=1,\ldots,n$. Clearly 3-stochastic cubic matrices are in a one-to-one correspondence to quadratic stochastic operators (see \ref{qso}).

\begin{remark}\label{remark on slices} Even being  $\textbf{P}$ cubic stochastic, neither its slices nor its fibers need to retain any stochastic property. Assume for instance that $\textbf{P}$ is cubic stochastic of type (1,2). Then only its frontal slices retain some stochastic nature. Indeed $\{\mathbf{P}_{::k}\}_{k=1}^n$ is a family of bidimensional probability distributions on the set of pairs $\{(i,j)\}_{i,j=1}^n$ (notice that  $\mathbf{P}_{::k}$  is nonnegative  and $\sum_{i,j=1}^n p_{ijk}=1$,   $k=1,\ldots,n$.)
 \end{remark}

\begin{remark}  The third order nature of cubic stochastic matrices increases the number of multiplications that can be defined for these matrices. See for instance \cite{bai}  where the authors consider  different associative multiplication rules defined on cubic matrices. Here  we are only interested in   multiplication rules that are consistent with the different definitions of stochasticity.
 \end{remark}

\subsection{Maksimov $\cdot$ multiplication.}\label{Maksimov multiplication}
Maksimov defined in \cite{m_cubic} the following multiplication for cubic stochastic matrices of type (1,2):
$$ (i,j,k)\cdot (m,r,s)=\delta_{km}(i,j,s) $$
where $(i,j,k)$ denotes the cubic matrix units and $\delta_{km}$ is the Kronecker's delta.
Elementwise, given $\textbf{A}=(a_{ijk})_{i,j,k=1}^n$, $\textbf{B}=(b_{ijk})_{i,j,k=1}^n$  and $\textbf{P}=(p_{ijk})_{i,j,k=1}^n$  such that
 $\textbf{P}=  \textbf{A} \cdot\textbf{B} $, we have
$$p_{ijs}=\sum_{k,r=1}^n a_{ijk}b_{krs}=\sum_{k=1}^n a_{ijk}b_{k+s}$$
where $b_{k+s}=\sum_{r=1}^n b_{krs}$.

\begin{proposition}\label{associativity cdot}   The set of cubic stochastic matrices of type (1,2) with the $\cdot$ multiplication forms a convex semigroup (i.e. a multiplicative associative semigroup such that $\lambda \textbf{A}+(1-\lambda)\textbf{B}$ is again cubic stochastic of type (1,2) for any cubic stochastic matrices $\textbf{A}$ and $\textbf{B}$ of type (1,2) and any $\lambda\in  \mathbb{R} $  such that $0\leq \lambda\leq 1$).
 \end{proposition}

\begin{proof}   The proof is analogous to that of \cite[Proposition 3]{m_cubic} taking into account that the associativity of $\cdot$ is a straightforward checking.
\end{proof}

\paragraph{} Some other  multiplications were also introduced in \cite{m_cubic} for different types of cubic stochastic matrices.  Elementwise descriptions of these multiplication rules can be found in \cite{ladra_flow,ladra_ACM}. The $\cdot$ multiplication   corresponds to  one of the  associative rules    in \cite{bai}.

\begin{remark} To better understand how the $\cdot$ multiplication acts on cubic stochastic matrices of type (1,2), let us recall the usual multiplication rule $(i,j)(l,k)=\delta_{jl}(i,k)$ on nonnegative column square (i.e. $n\times n$) matrices. Take $\textbf{P}=(p_{ij})_{i,j=1}^n\in NS(n,\mathbb{R})$  or  consider, equivalently, the Markov chain with transition probability matrix $\textbf{P}$. Then the $(i,j)$-th entry $p_{ij}$ of $\textbf{P}$ gives  the transition probability from a state $j$ to a state $i$ and the $(i,j)$-th entry
$p_{ij}^{(m)}$ of $\textbf{P}^{(m)}$, the $m$-th power of  $\textbf{P}$, the probability that a state $j$ reaches a state $i$ after $m$ transitions (i.e. $m$-step transition probabilities).
\end{remark}

\subsection{Probabilistic interpretation of the  $\cdot$ multiplication.}\label{Maksimov prob interpretation}
A probabilistic interpretation of the $\cdot$ multiplication appears in \cite{m_cubic} in terms of transition probabilities for walking particles. Here, however, we will give a different interpretation in a biological setting. (See  \cite{mukhamedov-paper,lyubich,paniello-LAA,paniello-MD}.) Consider a biological system (or population) i.e. a collection of organisms of types $\{1,2,\ldots,n\}$ closed with respect to reproduction, and let $p_{ijk}$ be the probability for a type $k$ individual to come from the mating of (ordered male and female) progenitors having types $i$ and $j$. Then:
\begin{enumerate}
\item[(i)] $\textbf{P}=(p_{ijk})_{i,j,k=1}^n$ is cubic stochastic of type (1,2) with
$$ p_{ijk}=P(father=i,mother=j\mid child=k). $$
\item[(ii)] The $(i,j,k)$-th entry   of $\textbf{P}^{(m,\cdot)}$, $m$-th power of $\textbf{P}$ with respect to the $\cdot $ multiplication, gives the probability for an individual  of type $k$ to appear after $m$
    matings from an initial ordered pair of ancestors of types $i$ and $j$.
\end{enumerate}
More details of this approach will be given in the final section.

\paragraph{}  Maksimov  reduced  cubic matrices to lower dimensional square $n\times n$ matrices obtained by adding the matrix entries along either the first or the second index.

\subsection{Accompanying matrices. Marginal distributions.}\label{def_accompanying}  Let
$\textbf{P}=(p_{ijk})_{i,j,k=1}^n$ be a cubic stochastic matrix of type (1,2). The following matrices were introduced in \cite{m_cubic}:
\begin{enumerate}
\item[(a)] The   first accompanying matrix  is
$\textbf{P}_1=(p_{i+k})_{i,k=1}^n$, with $p_{i+k} =\sum_{j=1}^n
p_{ijk}$.
\item[(b)] The  second accompanying matrix is    $\textbf{P}_2=(p_{+jk})_{j,k=1}^n$, with  $p_{+jk} =\sum_{i=1}^n
p_{ijk}$.
\end{enumerate}
These matrices were called $i$-accompanying and $j$-accompanying in \cite{m_cubic} (and denoted $ \textbf{P}_{(i)}$ and $ \textbf{P}_{(j)}$  respectively), as they were later in \cite{paniello-LAA} and \cite{paniello-MD}. However this notation, as first and second accompanying matrices, seems to better stress on how these matrices are formed. Clearly
 $\sum_{i=1}^n p_{i+k}=\sum_{i=1}^n\sum_{j=1}^np_{ijk}=\sum_{j=1}^np_{+jk}=1$, thus  both accompanying matrices are column stochastic.
 Hence, being also nonnegative, accompanying matrices are in $NS(n,\mathbb{R})$. We usually refer  to both accompanying matrices of a cubic stochastic matrix of type (1,2) to be the {\it marginal distributions} of cubic matrix \cite{paniello-MD}.

\subsection{(1,2)-transposes.}\label{(1,2)-transposes}
Let
 $\textbf{P}=(p_{ijk})_{i,j,k=1}^n$ be a cubic stochastic matrix of type (1,2).
 We define the (1,2)-{\sl transpose} of
$\textbf{P}$ and denote it by $\textbf{P}^{T(1,2)}$ to be the cubic matrix $\textbf{P}^{T(1,2)}=(q_{ijk})_{i,j,k=1}^n$ with $q_{ijk}=p_{jik}$, for all $i,j,k=1,\ldots,n$.  We will say that $\textbf{P}$ is (1,2)-{\sl symmetric} if $\textbf{P}=\textbf{P}^{T(1,2)}$, i.e. $p_{ijk}=p_{jik}$ for all $i,j,k=1,\ldots,n$.

Although actions of (semi)groups on cubic stochastic matrices will be discussed later, we notice here that (1,2)-transpositions   can be seen in terms of the action of $S_3$, the symmetric group of 3 elements, by permuting the matrix entries according to   permutations on the index triples $(i,j,k)$. Then $\textbf{P}^{T(1,2)}$ results from the action of the transposition $\sigma=(1,2) \in S_3$ as follows:
$$ \textbf{P}^{T(1,2)}=\sigma \textbf{P}=(p_{\sigma(i,j,k)})_{i,j,k=1}^n$$ where   $\sigma(i,j,k)=(j,i,k)$.

\begin{proposition}\label{P - PT csm}
 Let
 $\textbf{P}=(p_{ijk})_{i,j,k=1}^n$ be a cubic matrix. Then  $\textbf{P}=(p_{ijk})_{i,j,k=1}^n$ is stochastic  of type (1,2) if and only if so is
$\textbf{P}^{T(1,2)}$.
\end{proposition}

\begin{proof} If suffices to use the definition of cubic stochastic matrix of type (1,2) given in \ref{def_cubic stochastic}.
\end{proof}

\begin{lemma}\label{slices and transposes} Let
 $\textbf{P}=(p_{ijk})_{i,j,k=1}^n$ be  a cubic  stochastic matrix of type (1,2). Then:
 \begin{enumerate}
 \item[(i)] $(\textbf{P}^{T(1,2)})_{h::}=\textbf{P}_{:h:}$ for all $h=1,2,\ldots,n$.
  \item[(ii)]$(\textbf{P}^{T(1,2)})_{:h:}=\textbf{P}_{h::}$ for all $h=1,2,\ldots,n$.
   \item[(iii)] $(\textbf{P}^{T(1,2)})_{::h}=(\textbf{P}_{:h:})^T$ for all $h=1,2,\ldots,n$.
     \item[(iv)] If  $\textbf{P}$ is (1,2) symmetric, then its marginal distributions coincide.
  \end{enumerate}
\end{lemma}

\begin{proof} See Lemma 1 and Lemma 2 in \cite{paniello-EO}.
\end{proof}

\paragraph{ }
Marginal distributions provide a different approach to the study cubic stochastic matrices of type (1,2).
For instance, in \cite{paniello-MD}    partial solutions to the ergodicity of these  cubic matrices
were obtained by considering the bivariate Markov model defined by the marginal distributions.

\begin{theorem}\label{genetic BMC}{\rm \cite[Theorem 1]{paniello-MD}}  Let  $\textbf{P}$ be a   cubic stochastic matrix of type (1,2).
Then
$$\mathbf{X}_{t+1}=\left(%
\begin{array}{l}
 \mathbf{X}_{t+1}^{(1)}  \\
 \mathbf{X}_{t+1}^{(2)} \\
\end{array}
\right)
=\left(%
\begin{array}{cc}
  \lambda_{11} \textbf{P}^{(11)} & \lambda_{12}\textbf{P}^{(12)}  \\
 \lambda_{21}  \textbf{P}^{(21)} & \lambda_{22}\textbf{P}^{(22)}  \\
\end{array}
\right)%
\left(%
\begin{array}{l}
 \mathbf{X}_t^{(1)}  \\
 \mathbf{X}_t^{(2)} \\
\end{array}%
\right)$$
where
$\textbf{P}^{(11)}=\textbf{P}^{(12)} $ is    the  first accompanying matrix of $\textbf{P}$ and $\textbf{P}^{(21)} =\textbf{P}^{(22)} $  the   second accompanying matrix of $\textbf{P}$, provides a bivariate Markov model whenever $\lambda_{ij}\geq0$ for $1\leq i,j\leq 2$ and  $\sum_{j=1}^2\lambda_{ij}=1 $ for $i=1,2$.
\end{theorem}

\paragraph{ } Despite the fact that
 the relationship  given in Theorem \ref{genetic BMC}  between cubic stochastic matrices of type (1,2) and bivariate Markov matrices allows us to recover some information on the  cubic matrices (see for instance results on cubic matrices ergodicity in \cite[Theorem 2]{paniello-MD}), this connection is not enough to achieve a complete characterization of  cubic  matrices just looking at their marginal distributions. A similar lack of completeness follows when  one considers the already mentioned cubic matrix multiplications, as neither the above introduced $\cdot$ multiplication nor any of the remaining multiplication rules defined in \cite{m_cubic}, totally comprise the overall stochastic nature of the entries of a cubic stochastic matrix of type (1,2). To overcome   drawbacks of the $ \cdot$ multiplication,     $(i,j,k)\cdot (m,r,s)=\delta_{km}(i,j,s)$, only considering non-vanishing products (interactions) when the first index of the second (right) factor is involved and paying no attention to its middle index, the following multiplication was defined in \cite{paniello-MD} by weighting interactions involving also   middle  indices.

\subsection{$\star$ multiplication.}\label{star def}  The $\star$ product of two cubic stochastic matrices of type (1,2) is defined as follows:
$$ (i,j,k)  \star (m,r,s)
=
    \frac{1}{2}(\delta_{km} +  \delta_{kr} )(i,j,s).
$$

\begin{remark}\label{star remark}
  Notice that
\begin{align*}
 (i,j,k)  \star (m,r,s) &=  (i,j,k)  \cdot \Big(\frac{1}{2}(m,r,s)+ \frac{1}{2}(m,r,s)^{T(1,2)}  \Big) = \\
  &=  (i,j,k)  \cdot \Big(\frac{1}{2}(m,r,s)+ \frac{1}{2}(r,m,s)  \Big) = \\
&=   \frac{1}{2}(\delta_{km} +  \delta_{kr} )(i,j,s).
\end{align*}
Thus it comes out  that given  cubic stochastic matrices $\textbf{P}= (p_{ijk})_{i,j,k=1}^n$
and $\textbf{Q}= (q_{ijk})_{i,j,k=1}^n$    of type (1,2),   the $(i,j,k)$-th entry of
$\textbf{P} \star\textbf{Q}$ can be written as:
 \begin{align*} (\textbf{P} \star\textbf{Q})_{ijk}& =   \frac{1}{2} \Big(
 \sum_{r,s=1}^np_{ijr}q_{rsk}+ \sum_{r,s=1}^np_{ijr}q_{srk} \Big)
 = \frac{1}{2} \sum_{r =1}^n(  p_{ijr}q_{r+k}+ \sum_{r,s=1}^np_{ijr}q_{+rk})\\
 &=  \frac{1}{2} (p_{ij1},\ldots,p_{ijn})
 \left(
   \begin{array}{c}
      \left(
                                                   \begin{array}{c}
                                                     q_{1+k} \\
                                                     q_{2+k}  \\
                                                     \vdots \\
                                                    q_{n+k} \\
                                                   \end{array}
                                                 \right)
                                                 +
     \left(
                                                   \begin{array}{c}
                                                     q_{+1k} \\
                                                     q_{+2k}  \\
                                                     \vdots \\
                                                    q_{+nk} \\
                                                   \end{array}
                                                 \right) \\
   \end{array}
 \right)\\
      &= (\textbf{P}_{ij:})^T \Big(  \frac{1}{2} (\textbf{Q}_1)_{:k}  +\frac{1}{2}  (\textbf{Q}_2)_{:k} \Big).
  \end{align*}
 Hence $(\textbf{P} \star\textbf{Q})_{ijk}$ is a equally weighted product of the $ij$-tube $\textbf{P}_{ij:} $ of $\textbf{P}$ and the $k$-th columns of the  accompanying matrices of $\textbf{Q}$. (Here we extend the notation introduced in \ref{CM decomposition} to square matrices.)
 \end{remark}

\begin{proposition} Let $\textbf{P} $
and $\textbf{Q} $   be cubic stochastic matrices  of type (1,2). Then   $\textbf{P}  \star\textbf{Q} $ is again  cubic stochastic  of type (1,2).
\end{proposition}

\begin{proof} Clearly   $\textbf{P}  \star\textbf{Q} $ is a nonnegative matrix, hence it suffices to prove that $\sum_{i,j=1}^n (\textbf{P}  \star\textbf{Q})_{ijk}=1$ for all $k=1, \ldots,n$, but being both $\textbf{P} $
and $\textbf{Q} $   cubic stochastic of type (1,2) we have
  \begin{align*} \sum_{i,j=1}^n (\textbf{P}  \star\textbf{Q})_{ijk} &= \sum_{i,j=1}^n \frac{1}{2} \sum_{r,s=1}^n \big(p_{ijr}q_{rsk}+p_{ijr}q_{srk} \big)\\
&= \frac{1}{2}  \sum_{r,s=1}^n \Big( \big(\sum_{i,j=1}^n p_{ijr} \big)q_{rsk}+ \big(\sum_{i,j=1}^n p_{ijr} \big)q_{srk} \Big)\\
&=
\frac{1}{2}  \sum_{r,s=1}^n \big(  q_{rsk}+ q_{srk} \big)=1,
  \end{align*} which implies that $\textbf{P}  \star\textbf{Q} $ is again  cubic stochastic  of type (1,2).
\end{proof}

\begin{lemma}\label{star multiplication} The following statements hold:
\begin{enumerate} \item[(i)] The   $\star$ multiplication is associative. Moreover cubic stochastic matrices of type (1,2) form a  semigroup under $\star$.
\item[(ii)] Let $\textbf{P}$ and  $\textbf{Q} $  be   cubic stochastic matrices of type (1,2). If  $\textbf{Q} $  is (1,2)-symmetric, then
$\textbf{P}  \star\textbf{Q} = \textbf{P}   \cdot\textbf{Q} $.
\item[(iii)]  Let $\textbf{P}$ be a cubic stochastic matrix of type (1,2). If $\textbf{P}$  is (1,2)-symmetric, then the $m$-th powers $\textbf{P}^{(m,\cdot)}$  and $\textbf{P}^{(m,\star)}$ of $\textbf{P}$, with respect to $\cdot$ and $\star$ respectively, coincide.
\end{enumerate}
\end{lemma}
\begin{proof} The proof of (i) is just a case-by-case checking
 similar to that of Proposition \ref{associativity cdot}.  Write now $\textbf{P}= (p_{ijk})_{i,j,k=1}^n$
and $\textbf{Q}= (q_{ijk})_{i,j,k=1}^n$, Then if $q_{ijk}=q_{jik}$, for all $i,j,k=1,\ldots,n$, we have:
 \begin{align*}  (\textbf{P}  \star\textbf{Q})_{ijk} &=  \frac{1}{2} \sum_{r,s=1}^n \big(p_{ijr}q_{rsk}+p_{ijr}q_{srk} \big)\\
&= \frac{1}{2}  \sum_{r,s=1}^n \big(  p_{ijr}q_{rsk}+p_{ijr}q_{srk} \big) = \sum_{r,s=1}^n   p_{ijr}q_{rsk} =(\textbf{P}   \cdot\textbf{Q})_{ijk}
  \end{align*}
 Hence the (1,2)-symmetry of $\textbf{Q}$ ensures that $\textbf{P}  \star\textbf{Q} = \textbf{P}   \cdot\textbf{Q} $.

 Finally to prove (iii) it suffices to note that $m$-th powers of $\textbf{P}$ with respect to both multiplication rules are well-defined, as a result of the associativity, and apply then (ii).
 \end{proof}

\paragraph{} Clearly the $\star$ multiplication can be defined for arbitrary $n\times n\times n$ matrices. But being mainly concerned with stochastic matrices,  we have already restricted the definition to these matrices.

\subsection{${\rm CS_{(1,2)}}(n,\mathbb{R})$.}\label{CS semigroup}
We will denote by ${\rm CS_{(1,2)}}(n,\mathbb{R}) $ the set of all cubic stochastic matrices of type (1,2) with real entries and by
$({\rm CS_{(1,2)}(n,\mathbb{R})},\star)$ the semigroup formed with the $\star$ multiplication. Notice
it is a (right) monoid, as
  $\textbf{P}= (p_{ijk})_{i,j,k=1}^n$, with $p_{ijk}=1$ if $i=j=k$ and $p_{ijk}=0$ otherwise (see \cite[Figure 2.4]{kolda}) is cubic stochastic of type (1,2) and acts as a right (but not a left) identity element for
  $\star$.

\subsection{Weighted multiplications.}\label{WM}
The $\star$ multiplication is   a particular case of a more general weighted product
$$(i,j,k)  \star_{(\lambda_1,\lambda_2)} (m,r,s) =  (i,j,k)  \cdot \Big(\lambda_1(m,r,s)+ \lambda_2(m,r,s)^{T(1,2)}  \Big) $$
 for segregation coefficients $\lambda_1,\lambda_2\geq0$, such that $\lambda_1+\lambda_2=1$. Elementwise, given
$\textbf{P}, \textbf{Q}\in {\rm CS_{(1,2)}}(n,\mathbb{R})$,
 \begin{align*} (\textbf{P} \star_{(\lambda_1,\lambda_2)} \textbf{Q})_{ijk}& =
  \lambda_1\sum_{r,s=1}^n p_{ijr}q_{rsk}+  \lambda_2\sum_{r,s=1}^n p_{ijr}q_{srk}=\\
  &=\sum_{r,s=1}^n p_{ijr} ( \lambda_1q_{rsk}+  \lambda_2q_{srk})= \sum_{r=1}^n p_{ijr} ( \lambda_1q_{r+k}+  \lambda_2q_{+rk})\\
 &=  (p_{ij1},\ldots,p_{ijn})
 \left(
   \lambda_1\begin{array}{c}
      \left(
                                                   \begin{array}{c}
                                                     q_{1+k} \\
                                                     q_{2+k}  \\
                                                     \vdots \\
                                                    q_{n+k} \\
                                                   \end{array}
                                                 \right)
                                                 + \lambda_2
     \left(
                                                   \begin{array}{c}
                                                     q_{+1k} \\
                                                     q_{+2k}  \\
                                                     \vdots \\
                                                    q_{+nk} \\
                                                   \end{array}
                                                 \right) \\
   \end{array}
 \right)\\
      &=  (\textbf{P}_{ij:})^T \Big( \lambda_1 (\textbf{Q}_1)_{:k}  + \lambda_2(\textbf{Q}_2)_{:k} \Big) .
  \end{align*}
 Clearly $\star$ arises in the equally weighted  case $ \lambda_1= \lambda_2=\frac{1}{2}$, so that
 $\star=\star_{(\frac{1}{2},\frac{1}{2})}$, whereas Maksimov's $\cdot$ multiplication corresponds to the case $(\lambda_1,\lambda_2)=(1,0)$, i.e.
 $\cdot=\star_{(1,0)}$.

With a similar proof to that of Lemma \ref{star multiplication} the following results follow.

\begin{proposition}\label{weighted multiplication} Let $\lambda_1,\lambda_2$ be nonnegative real numbers such that $\lambda_1+\lambda_2=1$ and
$\textbf{P}, \textbf{Q} \in {\rm CS_{(1,2)}}(n,\mathbb{R})$. Then:

\begin{enumerate} \item[(i)] $\textbf{P} \star_{(\lambda_1,\lambda_2)} \textbf{Q}$ is again in  ${\rm CS_{(1,2)}}(n,\mathbb{R})$.
\item[(ii)] The  $ \star_{(\lambda_1,\lambda_2)}$ multiplication is associative and the set of   cubic stochastic matrices of type (1,2) form a  monoid  under the multiplication $ \star_{(\lambda_1,\lambda_2)}$.
\item[(ii)]   If  $\textbf{Q} $  is (1,2)-symmetric, then
$\textbf{P}  \star_{(\lambda_1,\lambda_2)}\textbf{Q} = \textbf{P}   \cdot\textbf{Q} $.
\item[(iii)]   If $\textbf{P}$  is (1,2)-symmetric, then
$\textbf{P}^{(m,\cdot)}=\textbf{P}^{(m,\star_{(\lambda_1,\lambda_2)})}$, where $\textbf{P}^{(m,\cdot)}$ and $\textbf{P}^{(m,\star_{(\lambda_1,\lambda_2)})}$ denote the $m$-th powers of  $\textbf{P}$  with respect to $\cdot$ and $\star_{(\lambda_1,\lambda_2)}$ respectively.
\end{enumerate}
\end{proposition}

\subsection{$({\rm CS_{(1,2)}}(n,\mathbb{R})  ,\star_{(\lambda_1,\lambda_2)})$.}
For any nonnegative $\lambda_1,\lambda_2$ such that$\lambda_1+\lambda_2=1$, we will denote by
$({\rm CS_{(1,2)}}(n,\mathbb{R})  ,\star_{(\lambda_1,\lambda_2)})$ the monoid formed by the set ${\rm CS_{(1,2)}}(n,\mathbb{R}) $ with the $\star_{(\lambda_1,\lambda_2)}$ multiplication. Recall that  the case $\lambda_1=\lambda_2=\frac{1}{2}$ corresponds to
$({\rm CS_{(1,2)}}(n,\mathbb{R})  ,\star )$.

\paragraph{ } Despite  the different weightings that can be considered on cubic matrices, our main interest will continue focused on the equally weighted $\star$ case. This product is the one arisen from the problem originally motivating our interest in cubic matrices, due to their role
when studying the transference of genetic inheritance in backwards Mendelian populations (see \cite{paniello-LAA,tian_coalgebras}).

 \paragraph{ } Products involving tensors having different dimensions were already considered in \cite{braman}, where   stress
 was on the action of $n\times n\times n$ tensors on $n\times n $ matrices. In the following section we will deal with an opposite problem, as we will consider actions on cubic stochastic matrices by square matrices.

\section{Actions on cubic stochastic matrices.}

In this section we   consider different actions of the semigroup  (of nonnegative column square  stochastic matrices)
 $NS(n,\mathbb{R})$  on the set
   ${\rm CS_{(1,2)}}(n,\mathbb{R})$ of  cubic stochastic matrices of type (1,2), paying also attention
   to how these actions behave on slices and marginal distributions of the cubic matrices.

\subsection{Group actions.} We recall that a group $(G,\circ)$ acts on a set $S$ if $g\cdot s\in S$ is defined for all $g\in G$ and $s\in S$ and we have
$e\cdot s=s$ and $(g_1\circ g_2)\cdot s=g_1\cdot (g_2 \cdot s)\in S$ for all $s\in S$ and $ g_1,g_2\in G$, where $e$ denotes the identity of $G$.

A similar definition holds if $G$ is a semigroup with identity element. All semigroups appearing in this section  have identity element. Notice that if $G$ is just a semigroup, as not every element $g\in G$ needs to have an inverse $g^{-1}$ in $G$, not all actions by elements of $G$ on $S$ will be reversible, but just those given by invertible elements.

\begin{theorem}\label{action def}
 Let $\{(i,j)\}_{i,j=1}^n$ and $\{(i,j,k)\}_{i,j,k=1}^n$ be the square and cubic matrix units. Then, extended by linearity,
 \begin{enumerate}
 \item[(i)]  $ (i,j)  \circledast_1  (r,s,t)=\delta_{jr} (i,s,t)  $
 \item[(ii)] $ (i,j)  \circledast_2  (r,s,t)=\delta_{js} (r,i,t)  $
 \end{enumerate}
 define actions of the   semigroup $NS(n,\mathbb{R})$  on the set  ${\rm CS_{(1,2)}}(n,\mathbb{R})$ of
 cubic stochastic matrices of type (1,2).
 \end{theorem}

\begin{remark}\label{action def remark 1}
We notice that $\circledast_2 $ can be obtained  from $\circledast_1 $ (and vice versa) by conjugation  given by the (1,2)-transposition.  Indeed, denoting by $T(1,2)$   the (1,2)-transposition on ${\rm CS_{(1,2)}}(n,\mathbb{R})$, we have
$$   (i,j)  \circledast_2  (r,s,t)=     \Big[(i,j)  \circledast_1 \big[  (r,s,t)  ^{T(1,2)} \big] \Big]^{T(1,2)}                 $$
thus $ \circledast_2 = T(1,2) \circ  \circledast_1 \circ T(1,2)$, or equivalently $T(1,2)\circ  \circledast_2 =\circledast_1 \circ T(1,2) $, considering  here (1,2)-transposes as the result of the action of the transposition $(1,2)\in S_3$ on the cubic matrix indices (see \ref{(1,2)-transposes}).
\end{remark}

\begin{remark}\label{action def remark 2} Actions
  $\circledast_1 $ and $\circledast_2 $
  can be alternatively defined using direct product notation as follows:
  \begin{enumerate}
 \item[(i)]  $ \big( (i,j)\times 1\times1  \big)\circledast_1  (r,s,t)=\delta_{jr} (i,s,t)  $
 \item[(ii)] $ \big( 1\times (i,j)\times1  \big) \circledast_2  (r,s,t)=\delta_{js} (r,i,t)  $
 \end{enumerate}
  where tensor product components act on the different indices of the triple $(r,s,t)$. Thus, given $\textbf{A}=(a_{ij})_{i,j=1}^n\in NS(n,\mathbb{R}) $
  and $\textbf{P}=(p_{ijk})_{i,j,k=1}^n\in{\rm CS_{(1,2)}}(n,\mathbb{R}) $, we have
    \begin{enumerate}
 \item[(i)]  $ \textbf{A} \circledast_1 \textbf{P }=\big( \textbf{A} \times \textbf{I}_n  \times \textbf{I}_n   \big)\cdot \textbf{P }= \big(  \sum_{r=1}^na_{ir}p_{rst}\big)_{i,s,t=1}^n $
 \item[(ii)] $ \textbf{A} \circledast_2 \textbf{P }=\big(  \textbf{I}_n  \times  \textbf{A} \times \textbf{I}_n  \big)\cdot   \textbf{P }= \big(  \sum_{s=1}^na_{is}p_{rst}\big)_{r,i,t=1}^n $
 \end{enumerate}
 where $ \textbf{I}_n  $ denotes the $n\times n$ identity matrix. Using this alternative approach to   $\circledast_1 $ and $\circledast_2 $,  as actions of
 $NS(n,\mathbb{R})\times NS(n,\mathbb{R})\times NS(n,\mathbb{R})$ on the set of cubic matrices,  Theorems \ref{action def} and \ref{action th2} follow easily. However detailed proofs are also included for completeness. Moreover, see Remark \ref{action def remark 1}, we have
 $\big(\textbf{A} \circledast_2 \textbf{P } \big)^{T(1,2)}=\textbf{A} \circledast_1 \big( \textbf{P }^{T(1,2)}\big)$, where, as noted before, $T(1,2)$ can be understood as the action of the transposition $(1,2)\in S_3$  on the cubic matrix indices.
\end{remark}

\begin{proof}
Let $\textbf{P}=(p_{ijk})_{i,j,k=1}^n$ be a cubic stochastic matrix of type (1,2) and let
 $\textbf{A}=(a_{ij})_{i,j=1}^n$  be in $NS(n,\mathbb{R})$ (so we have $\sum_{i=1}^na_{ij}=1$ for all $j=1,\ldots,n$). We first claim that both
     $  \textbf{A} \circledast_1 \textbf{P}$ and $  \textbf{A} \circledast_2 \textbf{P}$ are cubic stochastic of type (1,2). Indeed, if we write $  \textbf{A} \circledast_1 \textbf{P} =(q_{ist})_{i,s,t=1}^n$, then $  \textbf{A} \circledast_1 \textbf{P}$ has nonnegative entries  $q_{ist}=\sum_{r=1}^na_{ir}p_{rst}$ satisfying
 $$\sum_{i,s=1}^n q_{ist} =\sum_{i,s=1}^n\sum_{r=1}^n a_{ir}p_{rst} =\sum_{r,s=1}^n \big(\sum_{i=1}^n a_{ir}\big)p_{rst} =\sum_{r,s=1}^n p_{rst}=1.$$
 Hence
$  \textbf{A} \circledast_1 \textbf{P}$ is cubic stochastic of type (1,2).
Similarly we prove that $  \textbf{A} \circledast_2 \textbf{P} = (w_{rit})_{r,i,t=1}^n$
with $w_{rit}=\sum_{s=1}^n a_{is}p_{rst}$   is  cubic stochastic of type (1,2). (Proposition  \ref{P - PT csm} and Remark \ref{action def remark 1} also apply here.)

Clearly  $  \textbf{I}_n \circledast_1 \textbf{P}  = \textbf{P} =   \textbf{I}_n \circledast_2 \textbf{P}$ and, therefore, it only remains to be checked that
 $$ (i,j) \circledast_1 \Big[ (r,s)\circledast_1(t,u,v)   \Big] =\Big[ (i,j) (r,s) \Big]\circledast_1(t,u,v),  $$
 and
  $$ (i,j) \circledast_2 \Big[ (r,s)\circledast_2(t,u,v)   \Big] =\Big[ (i,j) (r,s) \Big]\circledast_2(t,u,v),  $$ hold.
Hence   both $\circledast_1$ and $\circledast_2$ define actions of $NS(n,\mathbb{R})$ on  ${\rm CS_{(1,2)}}(n,\mathbb{R})$.
\end{proof}

\paragraph{ }
Taking into account the semigroup structure of both
$NS(n,\mathbb{R})$ and  ${\rm CS_{(1,2)}(n,\mathbb{ R})}$ the next theorem follows.

\begin{theorem}\label{action th2}   $\circledast_1$  and $\circledast_2$  define
actions of  the semigroup $NS(n,\mathbb{R})$ of  nonnegative   column square stochastic matrices on
 $({\rm CS_{(1,2)}(n,\mathbb{ R})},\star)$.
\end{theorem}

\paragraph{ } The above defined actions  act independently on $({\rm CS_{(1,2)}(n,\mathbb{ R})},\star)$ and coincide on (1,2)-symmetric matrices.

\begin{proposition}\label{actions conmute}
Actions $\circledast_1$ and $\circledast_2$   commute    on the set of cubic  stochastic matrices of type (1,2), i.e.:
$$  (i,j)  \circledast_1\Big((k,l)  \circledast_2 (m,r,s) \Big) = (k,l)  \circledast_2 \Big ((i,j)  \circledast_1(m,r,s) \Big). $$
Hence $  \textbf{B}\circledast_2\big(\textbf{A}\circledast_1\textbf{P} \big)=\textbf{A}\circledast_1\big(\textbf{B}\circledast_2\textbf{P}\big)$ for any stochastic matrices $ \textbf{A},\textbf{B}\in  NS(n,\mathbb{R})  $   and $ \textbf{P}\in  {\rm CS_{(1,2)}(n,\mathbb{ R})}$.
\end{proposition}

\begin{proof}  Using Remark \ref{action def remark 2} it suffices to notice that $( \textbf{A}\times  \textbf{I}_n\times   \textbf{I}_n)(
 \textbf{I}_n\times  \textbf{B}\times \textbf{I}_n)=(
 \textbf{I}_n\times  \textbf{B}\times \textbf{I}_n)( \textbf{A}\times  \textbf{I}_n\times   \textbf{I}_n)$ in $NS(n,\mathbb{R})\times NS(n,\mathbb{R}) \times NS(n,\mathbb{R})$.
\end{proof}

\begin{corollary} Let    $\textbf{P} $  be a cubic stochastic matrix of type (1,2). If $\textbf{P} $  is (1,2)-symmetric, then
$\textbf{A}\circledast_2\textbf{P}=(\textbf{A}\circledast_1\textbf{P})^{T(1,2)}$ for any  $\textbf{A}\in NS(n,\mathbb{R})$.
\end{corollary}

\begin{proof}The result follows from Remark \ref{action def remark 1}.
\end{proof}

\paragraph{ }
Next we deal with the interaction between the $\circledast_1$  and $\circledast_2$  actions and the semigroup structure $({\rm CS_{(1,2)}(n,\mathbb{R})},\star)$  of the set of cubic stochastic matrices of type (1,2) under the $\star$ multiplication.

\begin{proposition}  Actions $\circledast_1$  and $\circledast_2$  are consistent with the $\star$ multiplication of cubic stochastic matrices of type (1,2):
\begin{enumerate}
\item[(i)] $  (i,j)\circledast_1 \Big((m,r,s)\star (p,q,k)   \Big) =\Big(   (i,j)\circledast_1  (m,r,s)     \Big)\star (p,q,k)$
\item[(ii)] $  (i,j)\circledast_2 \Big( (m,r,s)\star (p,q,k)   \Big) =\Big(   (i,j)\circledast_2   (m,r,s)     \Big)\star (p,q,k)$
\end{enumerate}
\end{proposition}

\begin{proof} The proof is a straightforward checking.
\end{proof}

\paragraph{} As noted in Remark \ref{remark on slices}, the family of frontal slices
$\{\textbf{P}_{::k}\}_{k=1}^n$ of a cubic stochastic matrix $\textbf{P}$ of type (1,2) gives rise to a family of bidimensional probability distributions. (Notice  this is in fact equivalent to $\textbf{P}$ being stochastic of type (1,2).) Actions
$\circledast_1$  and $\circledast_2$ act on the frontal slices of $\textbf{P}$  independently.

\begin{theorem}\label{actions on slices} Let $\textbf{P}$  be  a cubic stochastic matrix of type (1,2) and $\textbf{A}\in NS(n,\mathbb{R})$. Then the following statements hold:
\begin{enumerate}
\item[(i)] $\big( \textbf{A}\circledast_1  \textbf{P}\big)_{::k}=\textbf{A} \textbf{P}_{::k} $.
\item[(ii)] $\big( \textbf{A}\circledast_2  \textbf{P}\big)_{::k}=\textbf{A}\big( \textbf{P}_{::k}\big)^T= \textbf{A}\big( \textbf{P}^{T(1,2)}\big)_{::k} $.
\end{enumerate}
\end{theorem}

\begin{proof}
Let $\textbf{P}=(p_{ijk})_{i,j,k=1}^n$ and $\textbf{A}=(a_{ij})_{i,j=1}^n$. The above equalities follow from noticing that
$\textbf{A}\circledast_1  \textbf{P}=(q_{ijk})_{i,s,k=1}^n$ with $q_{isk}=\sum_{r=1}^n a_{ir}p_{rsk}$ and
$\textbf{A}\circledast_2 \textbf{P}=(w_{rik})_{r,i,k=1}^n$ with $w_{rik}=\sum_{s=1}^n a_{is}p_{rsk}$. The additional equality in (ii) follows from Lemma \ref{slices and transposes}(iii).
\end{proof}

\begin{corollary}\label{actions on matricization}  The matricization of $\textbf{A}\circledast_1  \textbf{P}$ by its frontal slices in
$$  \left(
      \begin{array}{c|c|c}
        \textbf{A}  \textbf{P}_{::1}    & \ldots\ldots  & \textbf{A} \textbf{P}_{::n}    \\
      \end{array}
    \right)
  $$
 and that of $\textbf{A}\circledast_2  \textbf{P} $ is
$$ \left(
      \begin{array}{c|c|c}
        \textbf{A}\big( \textbf{P}_{::1}\big)^T  &   \ldots\ldots & \textbf{A}\big( \textbf{P}_{::n}\big)^T  \\
      \end{array}
    \right)
 $$
\end{corollary}

\begin{proof} It  suffices to use Theorem \ref{actions on slices} and recall the matricization of  cubic matrices given in  \ref{CM decomposition}.
\end{proof}

\paragraph{ }
The role marginal distributions have to better understand cubic stochastic matrices
lead us to study how actions  $\circledast_1$ and $\circledast_2$
behave on their marginal distributions and also of the  derived bivariate Markov chain (see Theorem \ref{genetic BMC}).

\begin{proposition}\label{action MD}  Let $\textbf{P} $  be a cubic stochastic matrix of type (1,2).
 \begin{enumerate}
 \item[(i)]  The $\circledast_1$  action is consistent  first accompanying matrices, whereas
$\circledast_2$ is consistent with  second accompanying matrices, i.e., for any    $\textbf{A}\in NS(n,\mathbb{R}) $ we have:
\begin{enumerate}
\item[(a)] $\big(\textbf{A} \circledast_1 \textbf{P}\big)_1 =\textbf{A}  \textbf{P}_1  $,
\item[(b)] $\big(\textbf{A} \circledast_2 \textbf{P}\big)_2 =\textbf{A}  \textbf{P}_2  $.
\end{enumerate}
 \item[(ii)]     Second accompanying matrices
 remain  invariant  under   $\circledast_1$, whereas
    first accompanying matrices remain invariant under   $\circledast_2$, i.e.:
\begin{enumerate}
\item[(a)] $\big(\textbf{A} \circledast_1 \textbf{P}\big)_2 =   \textbf{P}_2  $,
\item[(b)] $\big(\textbf{A} \circledast_2 \textbf{P}\big)_1 = \textbf{P}_1  $.
\end{enumerate}
 \end{enumerate}
 \end{proposition}

\begin{proof}  Let
$\textbf{P}=(p_{ijk})_{i,j,k=1}^n$ be   cubic stochastic  of type (1,2)
and
 $\textbf{A}=(a_{ij})_{i,j=1}^n \in NS(n,\mathbb{R}) $.   Write  $ \textbf{A} \circledast_1 \textbf{P}  =(q_{ist})_{i,s,t=1}^n$.
To prove (i) we notice that
 $\big(\textbf{A} \circledast_1 \textbf{P}\big)_1 =\textbf{A}  \textbf{P}_1  $  follows from
$$ q_{i+t} =\sum_{s=1}^n q_{ist}=\sum_{s=1}^n \sum_{r=1}^n a_{ir}p_{rst}=
 \sum_{r=1}^n a_{ir}\big(\sum_{s=1}^n p_{rst}\big)=\sum_{r=1}^n a_{ir}p_{r+t} $$
 and  similarly one proves that
$\big(\textbf{A} \circledast_2 \textbf{P}\big)_2 =\textbf{A}  \textbf{P}_2  $ holds.

Next to prove   $\big(\textbf{A} \circledast_1 \textbf{P}\big)_2 =   \textbf{P}_2  $ it suffices to notice that
$$  q_{+st}=\sum_{i=1}^nq_{ist}=\sum_{i=1}^n\sum_{r=1}^n a_{ir} p_{rst}=\sum_{r=1}^n   p_{rst} \big(\sum_{i=1}^n  a_{ir}\big) =\sum_{r=1}^n   p_{rst}=   p_{+st},$$
since $\sum_{i=1}^n  a_{ir}=1$, for all $r=1,\ldots, n$. Therefore the second
 accompanying matrix of $ \textbf{A} \circledast_1 \textbf{P}  $ is exactly that of $\ \textbf{P}   $ and $\big(\textbf{A} \circledast_2 \textbf{P}\big)_1 = \textbf{P}_1  $ follows analogously.
\end{proof}

\begin{theorem}\label{action BMC}
Let  $\textbf{P}$ be a   cubic stochastic matrix of type (1,2) with associated bivariate Markov chain $\mathbf{X}_{t+1}= \mathbf{Q} \mathbf{X}_{t}$
where
$$\mathbf{Q}
=\left(%
\begin{array}{cc}
  \lambda_{11} \textbf{P}^{(11)} & \lambda_{12}\textbf{P}^{(12)}  \\
 \lambda_{21}  \textbf{P}^{(21)} & \lambda_{22}\textbf{P}^{(22)}  \\
\end{array}
\right)%
$$
being
$\textbf{P}^{(11)}=\textbf{P}^{(12)} $     the first accompanying matrix $\textbf{P}_1$ of $\textbf{P}$ and $\textbf{P}^{(21)} =\textbf{P}^{(22)} $  the   second accompanying matrix $\textbf{P}_2$ of $\textbf{P}$ and  $\lambda_{ij}\geq0$ for $1\leq i,j\leq 2$ and  $\sum_{j=1}^2\lambda_{ij}=1 $ for $i=1,2$.
For any nonnegative stochastic matrix $\textbf{A}\in NS(n,\mathbb{R})$:
\begin{enumerate}
\item[(i)] The $\circledast_1 $ action of  $\textbf{A}$ on  $\textbf{P}$ gives rise to the bivariate Markov chain  $\mathbf{X}_{t+1}= \mathbf{Q}_1 \mathbf{X}_{t}$
with
$$ \mathbf{Q}_1
=
\left(%
\begin{array}{cc}
    \textbf{A}  &   \textbf{0} \\
\textbf{0} &    \textbf{I}_n  \\
\end{array}\right)
\left(%
\begin{array}{cc}
  \lambda_{11} \textbf{P}^{(11)} & \lambda_{12}\textbf{P}^{(12)}  \\
 \lambda_{21}  \textbf{P}^{(21)} & \lambda_{22}\textbf{P}^{(22)}  \\
\end{array}
\right)%
$$
\item[(ii)] The $\circledast_2 $ action of  $\textbf{A}$ on  $\textbf{P}$ gives rise to the bivariate Markov chain  $\mathbf{X}_{t+1}= \mathbf{Q}_2\mathbf{X}_{t}$
with
$$\mathbf{Q}_2
=
\left(%
\begin{array}{cc}
   \textbf{I}_n  &   \textbf{0} \\
\textbf{0} &    \textbf{A}  \\
\end{array}\right)
\left(%
\begin{array}{cc}
  \lambda_{11} \textbf{P}^{(11)} & \lambda_{12}\textbf{P}^{(12)}  \\
 \lambda_{21}  \textbf{P}^{(21)} & \lambda_{22}\textbf{P}^{(22)}  \\
\end{array}
\right)%
$$
\item[(iii)] The simultaneous $\circledast_1 $ and $\circledast_2 $ actions of    $\textbf{A}$ on  $\textbf{P}$ give  rise to the bivariate Markov chain  $\mathbf{X}_{t+1}= \mathbf{Q}_3\mathbf{X}_{t}$
with
$$\mathbf{Q}_3
=
\left(%
\begin{array}{cc}
    \textbf{A}  &   \textbf{0} \\
\textbf{0} &    \textbf{A}  \\
\end{array}\right)
\left(%
\begin{array}{cc}
  \lambda_{11} \textbf{P}^{(11)} & \lambda_{12}\textbf{P}^{(12)}  \\
 \lambda_{21}  \textbf{P}^{(21)} & \lambda_{22}\textbf{P}^{(22)}  \\
\end{array}
\right)%
$$
\end{enumerate}
 \end{theorem}

\begin{proof}
(i) follows from  Proposition \ref{action MD} as $ \textbf{A} \textbf{P}^{(11)}= \textbf{A} \textbf{P}^{(12)}= \textbf{A} \textbf{P}_1 =\big(\textbf{A}\circledast_1 \textbf{P} \big)_1$ whereas
$ \big(\textbf{A}\circledast_1 \textbf{P} \big)_2=   \textbf{P}_2  $. Analogously one proves (ii) and (iii).
\end{proof}

\paragraph{} Results obtained in this section, besides on their own interest to study the structure of the semigroup of cubic stochastic matrices of type (1,2) through actions by $NS(n,\mathbb{R})$, will be applied in the following section in the particular setting of those populations arising in backwards Mendelian genetic inheritance.
The algebraic formulation of this problem is given in \cite{tian_coalgebras} (see also  \ref{Maksimov prob interpretation}).

\section{An example and further remarks.}

\paragraph{}  In this section, prior  to gathering some concluding remarks, we recover the problem
posed by Tian an Li \cite{tian_coalgebras} on the backwards   inheritance of the genetic  information in populations ruled by Mendel's laws.
Actions on cubic stochastic matrices are used to show how inheritance rules are modified when different mutations (i.e. changes modifying the inheritance probabilities) occur.

\subsection{Modeling mutations in genetic populations.}

\paragraph{}Let $\textbf{P}$ be a cubic stochastic matrix of type (1,2). In \cite{paniello-LAA} these cubic matrices were revisited in connection to coalgebras with genetic realization (or genetic coalgebras for short). A {\it coalgebra with genetic realization} is a real coalgebra $(C,\Delta)$, i.e. a finite dimensional real vector space $C$ with a basis ${\cal B}=\{e_1,\ldots, e_n\}$ and a linear map $\Delta:C\to C\otimes C$, called  {\it comultiplication}, such that $\Delta(e_k)=\sum_{i,j=1}^n \beta_{ij}^k e_i\otimes e_j$ with
\begin{enumerate}
\item[(i)] $0\leq \beta_{ij}^k\leq 1$, for all $i,j,k=1,\ldots,n$,
\item[(ii)] $\sum_{i,j=1}^n \beta_{ij}^k =1$ for all $k=1,\ldots,n$.
\end{enumerate}
Genetic coalgebras were introduced in \cite{tian_coalgebras} to model the backwards (from progeny to ancestors) genetic inheritance in Mendelian genetic systems.

\paragraph{}Looking at the  elements of (the natural basis) ${\cal B}$ as a complete set of representatives of the different types $\{1,2,\ldots,n\}$ for a given hereditary trait, we can identify $ \beta_{ij}^k=P(father=i, mother=j\mid child=k)   $.
Then $\textbf{P}=(p_{ijk})_{i,j,k=1}^n$  with $p_{ijk}=\beta_{ij}^k$, for all $i,j,k=1,\ldots,n$, is cubic stochastic of type (1,2) \cite{paniello-LAA}.

\paragraph{}Examples of genetic coalgebras, including a description of their related cubic stochastic matrices can be found in \cite{paniello-EO}. The reader is also referred to \cite{paniello-graphs} for examples of   graphical representations of  these systems by oriented graphs.

\paragraph{}One of the questions that remained unsolved in \cite{tian_coalgebras} was how to  introduce mutations
in Mendelian populations.
Mutations   have to be understood as changes affecting to the transition probabilities, (i.e. to the $\beta_{ij}^k$'s), so that the set $\{\beta_{ij}^k =p_{ijk}\}_{i,j,k=1}^n $ is replaced by a new set of transition probabilities $\{ \widetilde{p}_{ijk}\}_{i,j,k=1}^n $ still defining a cubic stochastic matrix.
 A first approach to this problem was given in \cite[Theorem 5.7]{tian_coalgebras}
by composing the coalgebra multiplication $\Delta$ with different linear maps having the property of leaving the simplex  $S^{n-1}$ (i.e. the set of probability distributions on the set $\{1,2,\ldots,n\}$) invariant.

\paragraph{} We notice that, due to the correspondence between genetic coalgebras and the elements of ${\rm CS_{(1,2)}(n,\mathbb{ R})}$
 \cite{paniello-LAA}, the backwards evolution of the corresponding genetic population can be studied using  cubic matrices. The $\star$ multiplication defined in \ref{star def} assumes paternal and maternal heritage to be symmetrical (i.e. equally weighted). Then changes in the inheritance probabilities (i.e. mutations) are introduced by having nonnegative column stochastic square ($ n\times n$) matrices, i.e. elements in $NS(n,\mathbb{R})$, acting on either the paternal ($\circledast_1  $) or maternal ($\circledast_2 $) heritage.

 \paragraph{}
 Take a fixed hereditary trait $k$ and consider the corresponding
  frontal $k$-slice  $\textbf{P}_{::k}$
of $\textbf{P}$. Then  $\textbf{P}_{::k}$ encloses the  probability distribution that traces the genetic inheritance from a trait $k$ back to its progenitors.  Now it  suffices to recall that having   $\textbf{A }\in NS(n,\mathbb{R})$  acting on  $\textbf{P}$, amounts to have
 $\textbf{A } $  acting on all $\textbf{P}_{::k}$, $k=1,\ldots,n$, simultaneously (see Proposition \ref{actions on slices}).

  \paragraph{} On the other hand, taking into account that the first accompanying matrix   $\textbf{P}_1=(p_{i+k})_{i,k=1}^n$
encloses the  probabilities $p_{i+k}=P(father=i\mid child=k)$, and that actions of $\textbf{A }\in NS(n,\mathbb{R})$ on $\textbf{P}$  result on actions (by multiplication) of  $\textbf{A}$ on the marginal distributions on $\textbf{P}$  (see Proposition \ref{action MD} and Theorem \ref{action BMC}), it becomes possible to consider mutations affecting   only
one of the two (paternal or maternal) inheritance lines. Similar results can be achieved when considering the second accompanying matrix $\textbf{P}_2$ of $\textbf{P}$ since then
$p_{+jk}=P(mother=j\mid child=k)$.

 \paragraph{} As a result the equally weighted $\star$ multiplication,  together to the actions $\circledast_1$ and $\circledast_2$, defined on ${\rm CS_{(1,2)}(n,\mathbb{ R})}$ provide a framework to consider different mutations in coalgebraic structures appearing in genetics.

\subsection{Further remarks.}

  \paragraph{}  In this paper we deal  with nonnegative stochastic ($n\times n$) square and ($n\times n\times n$) cubic matrices. We   first put aside their stochasticity to consider different algebraic properties
  of these sets of matrices. It was then possible, with suitable defined multiplications, to obtain different semigroup structures. We have also defined   different actions of the semigroup $NS(n,\mathbb{R})$ of nonnegative  column stochastic $n\times n$ matrices on the semigroup $( {\rm CS_{(1,2)}(n, \mathbb{R})}, \star)$ of cubic stochastic matrices of type (1,2). For these actions, aimed by the role matricization (i.e. tensor decomposition into matrices \cite{kolda}) and marginal distributions have to study these cubic matrices \cite{paniello-EO,paniello-graphs}, we have also considered how actions translate to their slices and related bivariate Markov chains.

  \paragraph{} Finally we have brought back the fact that our interest in cubic matrices stemmed  from the search for an algebraic framework where mutations  affecting to the transference of the information in backwards genetic inheritance (i.e. from progeny to ancestors) in Mendelian populations could be algebraically described. Keeping this in mind, we began this section briefly considering how mutations changing the inheritance probabilities    can be identified to actions on cubic stochastic matrices of type (1,2) by elements of
$NS(n,\mathbb{R})$.

 \paragraph{}  We remark that, although we have just focused on those cubic matrices linked to Mendelian inheritance,  similar approaches would also make sense for different types of stochastic tensors. It would  suffice to consider an acting (semi)group consistent with,
 both
 dimension and  stochastic properties of, the involved tensors.
  Consider,  for instance,  a 3-stochastic cubic matrix $\textbf{P}=(p_{ijk})_{i,j,k=1}^n$. Then for   any  permutation  $\sigma\in S_n$, i.e. any element of the symmetric group of $n$ elements,   $\sigma \textbf{P}= (p_{ij\sigma(k)})_{i,j,k=1}^n$
is again  3-stochastic.  It would therefore become possible to define an action of
 $S_n$ on the set of 3-stochastic cubic matrices by permuting the matrix frontal slices according to $\sigma$, or equivalently, actions on quadratic stochastic operators (see \ref{qso}).

   \paragraph{}
   This type  of actions  is expected
  to have reasonable applications to study changes in transition probabilities,   not only in biological populations but also in many  other types of systems or processes, as they provide an algebraic framework where displaying how these changes modify the future evolution of the system under consideration.

\section*{Acknowledgements.}
The author  wants to thank the referee for his/her carefully reading of the original manuscript and  also for all provided suggestions that have undoubtedly contributed to improve the paper. The author also wants to thank Professor M. Ladra for his comments on a preliminary version of the manuscript.

\end{document}